\documentclass{amsart}
\newtheorem{Theo}{Theorem}
\newtheorem{Lem}{Lemma}
\begin{document}
\title{Oscillations of the error term in the prime number theorem}
\author[J.-C. Schlage-Puchta]{Jan-Christoph Schlage-Puchta}
\begin{abstract}
Let $\sigma+i\gamma$ be a zero of the Riemann zeta function to the right of the line $\frac{1}{2}+it$. We show that this zero causes large oscillations of the error term of the prime number theorem. Our result is close to optimal both in terms of the magnitude and in the localization of large values for the error term.
\end{abstract}
\maketitle
\section{Introduction and Results}
Let $\Delta(x)=\Psi(x)-x$ be the error term in the prime number theorem. It is well known that the magnitude of $\Delta$ is related to the zeroes of the Riemann zetafunction, however, at first this relation appeared to be somewhat mysterious. Littlewood\cite{Littlewood} asked for a lower bound for $|\Delta(x)|$ depending only on a single zero of $\zeta$. In 1950, Tur\'an\cite{Turan oscillation} used his power sum method to prove the following.
\begin{Theo}
\label{thm:Turan}
If $\rho_0=\sigma_0+i\gamma_0$ is a zero of $\zeta$ with $\sigma_0\geq\frac{1}{2}$, and 
\[
X\geq\max(C_1, \exp(|\rho_0|^{60}),
\]
then
\[
\max_{x\leq X}|\Delta(x)|>\frac{X^{\sigma_0}}{|\rho_0|^{10\log X/\log_2 X}}\exp\left(-C_2\frac{\log X\log_3 X}{\log_2 X}\right),
\]
where $C_1, C_2$ are computable constants.
\end{Theo}
Pintz\cite{Pintz} improved both the localization and the lower bound for $\Delta$. He proved the following.
\begin{Theo}
\label{thm:Pintz}
If $\rho_0=\sigma_0+i\gamma_0$ is a zero of $\zeta$ with $\sigma_0\geq\frac{1}{2}$, and 
\[
X\geq\max(C_1, \gamma_0^{400}),
\]
then there exists $x\in[X^{1/4}, X]$ for which
\[
|\Delta(x)|>c_2\frac{x^{\sigma_0}}{\gamma_0^{50}},
\]
where $C_1, c_2$ are computable constants.
\end{Theo}
Here we prove the following.
\begin{Theo}
\label{thm:main}
If $\epsilon\in[0, 1]$ is a real number, $\rho_0=\sigma_0+i\gamma_0$ is a zero of $\zeta$ with $\sigma_0\geq\frac{1}{2}+\epsilon$ and $\gamma_0>5.5^{\frac{1}{\epsilon}}$, then for every 
\[
X\geq\max(C_1, \gamma_0^{12000\epsilon^{-3}}),
\]
there exists $x\in[X, X^{1+\epsilon}]$ such that 
\[
|\Delta(x)|>c_2\frac{x^{\sigma_0}}{\gamma_0^{1+\epsilon}},
\]
where $C_1, c_2$ are computable constants.
\end{Theo}
The constants 12000 and $5.5$ are not optimal, moreover, one can improve one of these constants at the expense of the other. However, the qualtitative dependence on $\epsilon$ cannot easily be improved.

Our proof follows Pintz' argument, the improvement comes from the fact that we replace the zero in question by another one, which gives at least the same lower bound, but is somewhat isolated. In this way we reduce the number of relevant zeroes, which reduces the loss when applying a power sum theorem. A similar approach has repeatedly been used by Tur\'an in connection with density estimates, see e.g. \cite{Turan density}.

The proofs of all three theorems mentioned here uses the power sum method, for a background we refer the reader to Tur\'an's book \cite{Turan book}.

\section{Proof}

The following was proven by Tur\'an \cite{Turan density}. As he remarked himself, the same result with $1.26$ in place of $0.71$ is a simple application of the three circle theorem. If one does not care about the constant $5.5$ in the theorem, any constant in place of $0.71$ would work.

\begin{Lem}
\label{Lem:Turan}
Let $\delta>0$, $\sigma>\frac{1}{2}+2\delta$, $T>T_0(\delta)$. Suppose that $\zeta$ has no roots in the region $\Re\;s\geq\sigma$, $|\Im\;s-T|<\log T$. Then $\zeta$ has $\leq 0.71\delta\log T$ zeroes in the square $\sigma-\delta\leq\Re\;s\leq\sigma$, $|\Im\;s-T|<\frac{\delta}{2}$.
\end{Lem}
If a zero of $\zeta$ satisfies the conditions of the lemma, we call it an exposed zero.

The following is a special case of the second main theorem of Tur\'an's theory of power sums as proven by Kolesnik and Straus\cite{KSTuran}.
\begin{Lem}
Let $z_1, \ldots, z_n$ be complex numbers. Then we have for every $m>0$ the bound
\[
\max_{m<\nu\leq m+n}\frac{1}{|z_1|^\nu}\left|\sum_{j=1}^nz_j^\nu\right| \geq\left(\frac{n}{4e(m+n)}\right)^n
\]
\end{Lem}

\begin{Lem}
\label{Lem:power sum}
Suppose that $\sigma_0+i\gamma_0$ is a zero of $\zeta$ with $\delta=\sigma_0-\frac{1}{2}>0$, and assume that $\gamma_0>\max(C, \log^{2/\delta}(1/\delta))$. Then there exists an exposed zero $\sigma_1+i\gamma_1$ of $\zeta$, such that $\sigma_1\geq\sigma_0$, $\gamma_1\leq2\gamma_0$.
\end{Lem}
\begin{proof}
Suppose there is no exposed zero in the rectangle $\Re\;s\geq\sigma_0$, $\Im s\leq 2\gamma_0$. Then there exists a sequence of roots $\sigma_j+i\gamma_j$, $0\leq j\leq n$, such that $\sigma_{j+1}\geq\sigma_j$, $|\gamma_j-\gamma_{j+1}|\leq \log 2\gamma_0$, and $\gamma_n\geq 2\gamma_0$. In particular we have $N(\sigma_0, 2\gamma_0)\geq\frac{2\gamma_0}{\log(2\gamma_0)}$. On the other hand we have $N(\sigma, T)\leq T^{\left(\frac{12}{5}+\epsilon\right)(1-\sigma)}$, thus, if $\gamma_0$ is sufficiently large we obtain a contradiction.
\end{proof}
Clearly, by decreasing the constant $c_2$ by a factor 2 at most we find that it suffices to prove Theorem~\ref{thm:main} for exposed zeros. 

Now fix $X$ and $\sigma_0+i\gamma_0$ as in Theorem~\ref{thm:main}, and let $\mu$ be a real number satisfying
\begin{equation}
\label{eq:mu def}
\left(1+\frac{\epsilon}{3}\right)\log X\leq\mu\leq \left(1+\frac{2\epsilon}{3}\right)\log X,
\end{equation}
and put $k=\frac{1}{40}\epsilon^2\log X$. As in Pintz' proof of Theorem~\ref{thm:Pintz}, we put $H(s)=\frac{\zeta'}{\zeta}(s)-\frac{1}{s-1}$, and compute 
\[
U=\frac{1}{2\pi i}\int_{2-i\infty}^{2+i\infty} H(s+i\gamma_1)e^{ks^2+\mu} s\;ds
\]
in two different ways. On one hand we can express $H(s)$ via $\Delta$, and obtain
\begin{eqnarray*}
U & = & \frac{1}{2\sqrt{\pi k}}\int_1^\infty\frac{\Delta(x)}{x}x^{i\gamma_0}\exp\left(-\frac{(\mu-\log x)^2}{4k}\right)\left(-\gamma_0+\frac{\mu-\log x}{2k}\right)dx\\
 & = & \frac{1}{2\sqrt{\pi k}}\left(\int_1^X + \int_X^{X^{1+\epsilon}} + \int_{X^{1+\epsilon}}^\infty\right),
\end{eqnarray*}
see\cite[(4.5)]{Pintz}. We now estimate the contribution of the first and the third integral. We have
\begin{eqnarray*}
\left|\int_1^X\right| & \leq & \left(\gamma_0+\frac{\mu}{2k}\right)
\int_1^X\exp\left(-\frac{(\mu-\log x)^2}{4k}\right)\;dx\\
 & \leq & 2\gamma_0 X \exp\left(-\frac{\epsilon^2\log^2 X}{36 k}\right)\\
 & \leq & 1,
\end{eqnarray*}
provided that $X>\gamma_0^{10}$.
Similarly we get
\begin{eqnarray*}
\left|\int_{X^{1+\epsilon}}^\infty\right| & \leq &\left(\gamma_0+\frac{\mu}{2k}\right)
\int_{X^{1+\epsilon}}^\infty\exp\left(-\frac{(\mu-\log x)^2}{4k}\right)\;dx\\
 & \leq & 2\gamma_0 \int_{(1+\epsilon)\log X}^\infty \exp\left(t-\frac{10}{3\epsilon}(t-\mu)\right)\;dt\\
  & \leq & 1.
\end{eqnarray*}
Finally, if $|\Delta(x)|<\frac{x^{\sigma_0}}{\gamma_0^{1+\epsilon}}$ for all $x\in[X, X^{1+\epsilon}]$, then 
\begin{eqnarray*}
\left|\int_X^{X^{1+\epsilon}}\right| & \leq & \frac{\gamma_0+\frac{20}{\epsilon}}{2\sqrt{\pi k}\gamma_1^{1+\epsilon}}\int_X^{X^{1+\epsilon}} x^{\sigma_0-1}\exp\left(-\frac{(\mu-\log x)^2}{4k}\right)\;dx\\
 & \leq & \frac{1}{\sqrt{\pi k}\gamma_1^\epsilon}\int_X^{X^{1+\epsilon}} x^{\sigma_0-1}\exp\left(-\frac{(\mu-\log x)^2}{4k}\right)\;dx\\
  & \leq & \frac{1}{\sqrt{\pi k}\gamma_1^\epsilon}\int_{\log X}^{(1+\epsilon)\log X}\exp\left(\sigma_0t-\frac{(\mu-t)^2}{4k}\right)\;dt\\
   & \leq & \frac{1}{\sqrt{\pi k}\gamma_1^\epsilon}\int_{-\infty}^{\infty}\exp\left(\sigma_0(\mu+r)-\frac{r^2}{4k}\right)\;dr\\
   & = & \frac{e^{\sigma_0\mu+\sigma_0^2k}}{\gamma_0^\epsilon}.
\end{eqnarray*}
Altogether we obtain that if the interval $[X, X^{1+\epsilon}]$ does not contain a large value of $\Delta$, then $|U|\leq 2\frac{e^{\sigma_0\mu+\sigma_0^2k}}{\gamma_0^\epsilon}$.

On the other hand we can express $U$ using complex integration and obtain
\[
U = \sum_\rho e^{k(\rho-i\gamma_0)^2+\mu(\rho-i\gamma_0)} + \mathcal{O}(1),
\]
where the sum runs over non-trivial zeros of $\zeta$, see \cite[(5.1)]{Pintz}. We divide the zeros of $\zeta$ occurring in this sum in four sets:  Let $Z_1$ be the set of zeros $\rho_1=\sigma_1+\gamma_1$ with $|\gamma_0-\gamma_1|>\log \gamma_0$, $Z_2$ be the set of zeros satisfying $\frac{\epsilon}{16}\leq |\gamma-\gamma_0|\leq\log \gamma_0$, $Z_3$ be the set of zeros in the rectangle
\[
0\leq\Re\rho\leq\sigma_0-\frac{\epsilon}{16},\quad |\gamma-\gamma_0|\leq\frac{\epsilon}{16},
\]
and $Z_4$ be the set of zeros satisfying $|\gamma_1-\gamma_0|\leq\frac{\epsilon}{16}$ and $\sigma_1\geq\sigma_0-\frac{\epsilon}{16}$. In view of Lemma~\ref{Lem:Turan} we see that $|Z_4|\leq 0.71\frac{\epsilon}{8}\log \gamma_0\leq \frac{\epsilon}{11}\log\gamma_0$.

We will first show that only the contribution of $Z_4$ is relevant.

We have
\begin{eqnarray*}
\left|\sum_{\rho\in Z_1} e^{k(\rho-i\gamma_0)^2+\mu(\rho-i\gamma_0)}\right| & \leq & 
e^{\sigma_0\mu} \sum_{\rho\in Z_1} e^{-k|\rho-\rho_0|^2/2}\\
 & \leq & 2e^{\sigma_0\mu} e^{-\frac{k}{2}\log^2\gamma_0}\gamma_0\log\gamma_0 +  2e^\mu\sum_{\gamma>\gamma_0+\log\gamma_0}e^{-k(\gamma-\gamma_0)^2/2}\\
  & \leq & e^{\sigma_0\mu} + 2e^{\sigma_0\mu} \sum_{n\geq \log\gamma_0} e^{-kn^2/2} \log(n+\gamma_0)\\
  & \leq & 2e^{\sigma_0\mu}\\
   & \leq & 2e^{\sigma_0\mu+\sigma_0k^2} X^{-\frac{\epsilon^4}{3200}},
\end{eqnarray*}
since  $\sigma_0\geq\frac{1}{2}$.

If $\rho\in Z_2$, then $\Re\;\rho\leq\sigma_0$, thus
\[
\Re\;(\rho-i\gamma_0)^2 = \sigma^2-(\gamma-\gamma_0)^2 \leq \sigma_0^2-\frac{\epsilon^2}{256},
\]
and therefore
\begin{eqnarray*}
\left|\sum_{\rho\in Z_2} e^{k(\rho-i\gamma_0)^2+\mu(\rho-i\gamma_0)}\right| & \leq &
\big(N(\gamma_0+\log \gamma_0)-N(\gamma_0-\log \gamma_0)\big)e^{\sigma_0\mu+\sigma_0^2k - \frac{\epsilon^2}{256}k}\\
 & \leq & 2\log^2\gamma_0 e^{\sigma_0\mu+\sigma_0^2k} X^{-\frac{\epsilon^4}{11000}}
\end{eqnarray*}

If $\rho\in Z_3$, then $\Re\rho\leq\sigma_0-\frac{\epsilon}{16}$, and $|Z_3|\leq \log\gamma_0$, thus
\[
\sum_{\rho\in Z_3}e^{k(\rho-i\gamma_0)^2+\mu(\rho-i\gamma_0)} \leq \log\gamma_0 e^{\sigma_0\mu+\sigma_0^2 k -\frac{\epsilon\mu}{16}} \leq \log\gamma_0 e^{\sigma_0\mu+\sigma_0^2 k}X^{-\epsilon/16}
\]

Altogether we obtain
\begin{eqnarray*}
\sum_{\rho\not\in Z_4} e^{k(\rho-i\gamma_0)^2+\mu(\rho-i\gamma_0)} & \leq &  2\log^2\gamma_0 e^{\sigma_0\mu+\sigma_0^2 k}\left(X^{-\frac{\epsilon^4}{3200}} + X^{-\frac{\epsilon^4}{11000}} + X^{-\frac{\epsilon}{8}}\right)\\
 & \leq & \frac{e^{\sigma_0\mu+\sigma_0^2 k}}{\gamma_0^\epsilon},
\end{eqnarray*}
provided that $X>\gamma_0^{12000\epsilon^{-3}}$. We conclude that if $|\Delta(x)|$ is small throughout the interval$[X, X^{1+\epsilon}]$, then 
\[
\left|\sum_{\rho\in Z_4} e^{k(\rho-i\gamma_0)^2+\mu(\rho-i\gamma_0)}\right| \leq 3\frac{e^{\sigma_0\mu+\sigma_0^2k}}{\gamma_0^\epsilon}
\]
holds for all $\mu$ satisfying (\ref{eq:mu def}). We now apply Lemma~\ref{Lem:power sum} and find that for some $k$ in this range, the left hand side of this inequality is bounded below by $e^{\sigma_0\mu+\sigma_0^2k}$ multiplied by
\[
\left(\frac{1}{4e\left(\frac{3}{\epsilon}+3\right)}\right)^{|Z_4|} \geq 45^{-\frac{1}{11}\epsilon\log\gamma_0}\geq 1.42^{-\varepsilon\log\gamma_0} \geq \gamma_0^{-0.35\epsilon},
\]
thus $\gamma_0^\epsilon<3\gamma_0^{0.35\epsilon}$, which inpliess $\gamma_0<5.5^{\frac{1}{\epsilon}}$, contrary to our assumption.

\end{document}